\documentclass[preprint,12pt,number]{elsarticle}

\usepackage{amsmath,amssymb,amsthm,mathtools}
\usepackage[T1]{fontenc}
\usepackage{lmodern}

\newtheorem{theorem}{Theorem}[section]

\newtheorem{lemma}[theorem]{Lemma}

\newtheorem{corollary}[theorem]{Corollary}
\newtheorem{remark}[theorem]{Remark}
\newtheorem{definition}[theorem]{Definition}

\newcommand{\D}{\mathbb D}

\newcommand{\E}{\mathbb E}
\newcommand{\eps}{\varepsilon}

\journal{Journal of Functional Analysis}

\begin{document}

\begin{frontmatter}

\title{Randomized second order Riesz projections on the Hamming cube}

\author[pku]{Yiming Chen}
\ead{ymchenmath@math.pku.edu.cn}

\author[hkust]{Guozheng Dai}
\ead{guozhengdai@ust.hk}

\address[pku]{School of Mathematical Sciences, Peking University}
\address[hkust]{Department of Mathematics, Hong Kong University of Science and Technology}

\begin{abstract}

In this paper, we improve the arbitrary Banach space \(n \log n\) bound of Ivanisvili--Volberg \cite{IvanisviliVolberg2022} for the  second order projection bound to the order \(\sqrt{n}\) bound.

Moreover, we study the lower Riesz estimate with the pointwise square gradient, and prove a fixed chaos characterization: on every fixed homogeneous Walsh chaos $H_k$, the dimension free estimate
\[
   \|\Delta^{1/2}f\|_{L^p(\Omega_n;X)}
   \lesssim_{p,k,X}
   \||\nabla f|_X\|_{L^p(\Omega_n)}
\]
holds for all $n$ if and only if $X$ has Rademacher type $2$.

We also consider an exact tail space norm of the analytic paraproduct $T_\varphi g(z)=\int_0^z g(\zeta)\varphi'(\zeta)\,d\zeta$ on Banach valued \(H^\infty\) spaces. A matching lower bound of Volberg \cite{Volberg2024}

\[
\|T_\varphi:H_d^\infty(\mathbb D;Y)\to H^\infty(\mathbb D;Y)\|
\asymp_{\alpha,\varphi} d^{-\alpha}
\]
under a nondegenerate boundary singularity assumption is established.

\end{abstract}

\end{frontmatter}

\section{Introduction}

Analysis on the discrete cube $\Omega_n=\{-1,1\}^n$  provides a model in which harmonic analysis, probability, and Banach space geometry meet in a particularly transparent way.  The Walsh characters diagonalize the number operator $\Delta$, the heat semigroup $e^{-t\Delta}$ has an elementary probabilistic representation, and many analytic inequalities reduce to questions about the interaction between Fourier-Walsh multipliers and random signs.  In the scalar setting this circle of ideas is closely connected with hypercontractivity, logarithmic Sobolev inequalities, and concentration on product spaces; see, for instance, the classical works of Beckner, Bonami, and Gross, as well as the Poincaré and concentration inequalities of Bobkov--G\"otze, Talagrand, and Ben-Efraim--Lust-Piquard \cite{Beckner1975,Bonami1970,GrossLogSobolev1975,GrossClifford1975,BobkovGotze1999,Talagrand1995,BenEfraimLustPiquard2008}.  For background on the Fourier-Walsh calculus on the cube we refer to \cite{ODonnell2014}.

The Banach-valued theory is more delicate because estimates that are dimension free for scalar- or Hilbert-valued functions may depend on the geometry of the target space.  Type, cotype, $K$-convexity, and related martingale properties enter naturally in this problem, the relevant Banach-space background goes back to Kwapien, Maurey--Pisier, Kahane, and Pisier, and is systematically studied in \cite{Kahane1985,Kwapien1972,LedouxTalagrand1991,MaureyPisier1976,Pisier1986,Pisier2016}.  On the cube and in closely related discrete settings, Riesz transforms associated with the number operator were studied by Meyer and Lust-Piquard, with later developments for discrete groups and second order transforms in \cite{Meyer1984,LustPiquard1998,LustPiquard2004,DomelevoPetermichl2014}.  More recently, Pisier-type inequalities on the discrete cube were refined by Hyt\"onen--Naor and by Ivanisvili--van Handel--Volberg, while Ivanisvili--Volberg developed Banach-valued Riesz and Pisier inequalities for singular integrals on the Hamming cube \cite{HytonenNaor2013,IvanisviliVanHandelVolberg2020,IvanisviliVolberg2022}.  Volberg's subsequent work on tail spaces and Bernstein--Markov inequalities provides the analytic background for several of the questions considered here \cite{Volberg2024}.

The main result of this paper is a $\sqrt n$ upper bound for randomized second order Riesz projections with arbitrary Banach space.  More precisely, for every Banach space $X$, every $1\le p<\infty$, every $n\ge1$, and every $g:\Omega_n\to X$, we prove

\begin{equation}\label{equ-main-intro-1}
   \left(\E_{\eps,\delta}\left\|
 \sum_{j=1}^n \delta_j\Delta^{-1}D_jg(\eps)
 \right\|_X^p\right)^{1/p}
 \le C_p\sqrt n\,\|g\|_{L^p(\Omega_n;X)}. 
\end{equation}
 
Here $D_j$ are the Walsh derivatives, $\Delta=\sum_{j=1}^nD_j$, and $\Delta^{-1}$ is taken on the mean zero Walsh subspace.  By duality, when $1<p<\infty$ \eqref{equ-main-intro-1} gives the corresponding second order inequality for vector fields, and improves general Banach-space upper bound of order $n\log n$ in Ivanisvili--Volberg \cite{IvanisviliVolberg2022} for the same estimate, and \eqref{equ-main-intro-1} attains the order suggested by the known worst case behavior of the Hyt\"onen--Naor example \cite{HytonenNaor2013,IvanisviliVolberg2022}.  The proof uses only a biased sign heat representation of $P_tD_j$ together with scalar moment estimates.

The same method gives an anisotropic form.  Given positive weights $a_1,\ldots,a_n$, let $\Delta_a W_A=\left(\sum_{j\in A}a_j\right)W_A.$ We prove a weighted upper bound for the operators $a_j\Delta_a^{-1}D_j$ with a heat kernel constant $\Lambda_p(a)$ depending on the distribution of the time scales $a_j^{-1}$.  In particular, for $1\le p\le2$ the constant is bounded by
\[
 \Lambda_2(a)=\int_0^\infty
 \left(\sum_{j=1}^n \frac{a_j^2}{e^{2a_jt}-1}\right)^{1/2}\,dt,
\]
which equals $(\pi/2)\sqrt n$ in the isotropic case.  The result shows that the inverse $\Delta_a^{-1}$ couples all active frequencies through the weighted sum $\sum_{j\in A}a_j$.% is not merely a coordinate-subset refinement of the unweighted theory; 

%Secondly, we also show that any lower Riesz estimate of the form
%\[
% \|\Delta^{1/2}f\|_{L^p(\Omega_n;X)}
% \le C_p\,\||\nabla f|_X\|_{L^p(\Omega_n)}
%\]
%forces the target space $X$ to have Rademacher type $2$.  This obstruction is already present on every fixed homogeneous Walsh chaos, so it is not a low-degree artifact, it may be viewed as an interesting comparison with the finite cotype condition in \cite{IvanisviliVolberg2022}.  

The second result concerns lower Riesz estimates.  Ivanisvili--Volberg \cite{IvanisviliVolberg2022} consider a lower Riesz problem with a randomized gradient,
\[
   \left(\mathbb E_\delta
   \left\|\sum_{j=1}^n\delta_jD_jf\right\|_{L^p(\Omega_n;X)}^p
   \right)^{1/p},
\]
and ask whether finite cotype may suffice for that estimate \cite{IvanisviliVolberg2022}.  In this paper, we consider the square gradient quantity discussed in Volberg \cite{Volberg2024}
\[
   \||\nabla f|_X\|_{L^p(\Omega_n)}
   =
   \left\|
   \left(\sum_{j=1}^n\|D_jf\|_X^2\right)^{1/2}
   \right\|_{L^p(\Omega_n)}.
\]

For this square gradient version we prove a sharp fixed chaos result.  If $H_k(\Omega_n;X)$ denotes the $k$-th homogeneous Walsh chaos, then for every fixed $k\ge1$, $f\in H_k(\Omega_n;X)$ and $1<p<\infty$ the best dimension free constant in
\[
   \|\Delta^{1/2}f\|_{L^p(\Omega_n;X)}
   \le C\,\||\nabla f|_X\|_{L^p(\Omega_n)},
\]
is equivalent, up to constants depending only on $p$ and $k$, to the Rademacher type $2$ constant $T_2(X)$.  Thus type $2$ is the Banach space condition on every fixed homogeneous Walsh chaos.  %The lower bound is the elementary block construction already visible in the first chaos, while the converse uses the Banach-valued decoupling theorem of de la Pe\~na and Montgomery-Smith for Rademacher chaoses \cite{delaPenaMontgomerySmith1995}.

%This gives a precise interpretation of the finite-cotype obstruction.  Since spaces such as $\ell_r$, $1<r<2$, have finite cotype, indeed cotype $2$, but fail to have Rademacher type $2$, finite cotype cannot imply the pointwise square-gradient lower Riesz estimate even on bounded homogeneous levels.  However, this statement should not be read as a negative solution of the randomized-gradient lower Riesz problem of Ivanisvili--Volberg; the randomized and square-gradient right-hand sides are different in general Banach spaces.

Moreover, we prove a sharp paraproduct tail estimate under the standard analytic corner asymptotics.  Namely, for the operator $T_\varphi g(z)=\int_0^z g(\zeta)\varphi'(\zeta)\,d\zeta$ acting on the tail space $H_d^\infty(\mathbb D;Y)$, the norm is of order $d^{-\alpha}$ when $\varphi$ has a corner singularity of exponent $\alpha\in(0,1)$.  The proof uses classical one variable function theory and the Lehman--Warschawski corner expansion, in the spirit of the analytic framework of Coifman--Meyer and standard Hardy-space references \cite{CoifmanMeyer1978,Duren1970,Garnett2007,Lehman1957,Pommerenke1992,Warschawski1942}.  %Consequently, the analytic paraproduct step in Volberg's proof cannot by itself improve the decay from $d^{-\alpha}$ to $d^{-1}$ when $\alpha<1$.

Finally, we consider a structural iteration principle for Bernstein--Markov estimates on the cube. A first order result bounds the square gradient $|\nabla f|_X$ for Banach valued Walsh polynomials under type assumptions on $X^*$ and uses Pisier's holomorphic-semigroup machinery \cite{Pisier1982,Volberg2024}.  We show that once such a first order estimate is available uniformly for finite Hilbertian sums $\ell_2^M(X)$, it automatically iterates to higher order Hilbertian square functions.  Thus no additional Banach space geometry is needed to pass from the first-order Bernstein--Markov estimate to the ordered and distinct $k$-th order square gradient estimates, the cost is the $k$-fold product of the first order constant.

The paper is organized as follows.  Section 2 introduces the preliminary and notations used throughout.  Section 3 proves the $\sqrt n$ estimate for randomized second order Riesz projections and derives the dual form, and establishes the anisotropic weighted extension.  Section 4 proves the type $2$ obstruction for lower Riesz estimates, including a fixed homogeneous chaos version.  Section 5 proves the sharp paraproduct tail estimate for analytic corner maps, and proves the Bernstein--Markov iteration principle.

\section{Preliminaries on the Hamming cube}
\label{sec:preliminaries}

We introduce the notation and basic facts used throughout the paper.  All finite cubes are equipped with their uniform probability measures, and all expectations are normalized.  Biased product measures used for heat-semigroup representations will be specified separately.  The letter $C$ denotes a finite positive constant whose value may change from line to line.  We write $A\lesssim B$ if $A\leq CB$, and $A\asymp B$ if both $A\lesssim B$ and $B\lesssim A$ hold. %; subscripts indicate the parameters on which it may depend

\subsection{Walsh expansions and discrete derivatives}

For an integer $n\geq1$, let $\Omega_n:=\{-1,1\}^n$, equipped with the uniform probability measure.  We write elements of $\Omega_n$ as
$\eps=(\eps_1,\ldots,\eps_n)$, and expectations with respect to $\eps$ as $\E_\eps$.  If $A\subseteq\{1,\ldots,n\}$, the corresponding Walsh character is
\[
 W_A(\eps)=\prod_{j\in A}\eps_j,
 \qquad W_\varnothing\equiv1.
\]
For a Banach space $X$, every function $f:\Omega_n\to X$ has the Fourier--Walsh expansion
\[
 f=\sum_{A\subseteq\{1,\ldots,n\}} \widehat f(A)W_A,
\]
where$ \widehat f(A)=\E_\eps W_A(\eps)f(\eps)\in X.$ The degree of $f$ is the largest $|A|$ for which $\widehat f(A)\neq0$.  

We denote by
\[
 \mathcal P_d(\Omega_n;X)
 =\{f:\Omega_n\to X:\widehat f(A)=0\text{ whenever }|A|>d\}
\]
the space of $X$-valued Walsh polynomials of degree at most $d$, and by
\[
 \mathcal H_k(\Omega_n;X)
 =\{f:\Omega_n\to X:\widehat f(A)=0\text{ whenever }|A|\neq k\}
\]
the $k$-th homogeneous Walsh chaos.

For $1\leq j\leq n$, the Walsh derivative $D_j$ is the Fourier multiplier
\[
 D_j W_A={\bf 1}_{\{j\in A\}}W_A.
\]
Equivalently, if $\eps^{(j)}$ denotes the point obtained from $\eps$ by flipping the $j$-th coordinate, then
\[
 D_jf(\eps)=\frac{f(\eps)-f(\eps^{(j)})}{2}.
\]
Thus $D_j^2=D_j$, the operators $D_i$ and $D_j$ commute, and for
$A\subseteq\{1,\ldots,n\}$ we shall write $D_A=\prod_{j\in A}D_j,$ and $D_\varnothing=I.$ With this convention,
\[
 D_AW_B={\bf 1}_{\{A\subseteq B\}}W_B.
\]
In particular, the Walsh multipliers $D_j$ and $D_A$ never increase Walsh degree.

\subsection{Number operators, inverse operators, and heat semigroups}

The number operator, or Hamming cube Laplacian, is $\Delta=\sum_{j=1}^nD_j,$ then $\Delta W_A=|A|W_A.$ We define $\Delta^{1/2}W_A=|A|^{1/2}W_A\quad(A\neq\varnothing),$ and $\Delta^{1/2}W_\varnothing=0,$ and interpret $\Delta^{-1}$ as the inverse of $\Delta$ on the mean zero Walsh subspace and as zero on constants:
\[
 \Delta^{-1}W_A=|A|^{-1}W_A\quad(A\neq\varnothing),
\]
and $$\Delta^{-1}W_\varnothing=0.$$

Thus, if $f\in\mathcal H_k(\Omega_n;X)$ with $k\geq1$, then $\Delta f=kf, \Delta^{1/2}f=\sqrt{k}\,f.$ We also use weighted number operators.  Given weights
$a=(a_1,\ldots,a_n)\in(0,\infty)^n$, set $\Delta_a=\sum_{j=1}^n a_jD_j,$ $\Delta_a W_A=a(A)W_A,$ $a(A)=\sum_{j\in A}a_j.$ Again,
\[
 \Delta_a^{-1}W_A=a(A)^{-1}W_A\quad(A\neq\varnothing),
 \qquad
 \Delta_a^{-1}W_\varnothing=0.
\]
For $t\geq0$, the unweighted and weighted heat semigroups are $P_t=e^{-t\Delta}, P_t^a=e^{-t\Delta_a},$ so that
\[
 P_tW_A=e^{-t|A|}W_A,
 \qquad
 P_t^aW_A=e^{-t a(A)}W_A.
\]

The heat semigroup has a useful biased-sign representation.  For $t>0$, let
$\xi(t)=(\xi_1(t),\ldots,\xi_n(t))$ be independent signs with
\[
 \mathbb P\{\xi_j(t)=1\}=\frac{1+e^{-t}}2,
 \qquad
 \mathbb P\{\xi_j(t)=-1\}=\frac{1-e^{-t}}2.
\]
Then
\[
 P_tf(\eps)=\E_\xi f(\eps\xi(t)),
\]
where multiplication is coordinatewise.  More generally, for the weighted semigroup we choose independent signs $\xi^a_j(t)$ satisfying
\[
 \mathbb P\{\xi^a_j(t)=1\}=\frac{1+e^{-a_jt}}2,
 \qquad
 \mathbb P\{\xi^a_j(t)=-1\}=\frac{1-e^{-a_jt}}2,
\]
so that $\mathbb E\xi^a_j(t)=e^{-a_jt}$.  With
$\xi^a(t)=(\xi^a_1(t),\ldots,\xi^a_n(t))$, one has
\[
 P_t^af(\eps)=\E_\xi f(\eps\xi^a(t)).
\]
In particular, $P_t$ and $P_t^a$ are contractions on $L^p(\Omega_n;X)$ for $1\leq p\leq\infty$.

We repeatedly use the following derivative form of the same representation.

\begin{lemma}
\label{lem:biased-sign-derivative}
Let $a_1,\ldots,a_n>0$, $t>0$, and $1\leq j\leq n$.  Put $Z_j^a(t)=\frac{\xi_j^a(t)-e^{-a_jt}}{\sqrt{1-e^{-2a_jt}}}.$ Then $\mathbb E Z_j^a(t)=0$, $\mathbb E|Z_j^a(t)|^2=1$, and for every Banach space $X$ and every $g:\Omega_n\to X$,
\[
 P_t^aD_jg(\eps)
 =\frac{e^{-a_jt}}{\sqrt{1-e^{-2a_jt}}}
 \E_\xi\bigl[Z_j^a(t)g(\eps\xi^a(t))\bigr].
\]
In the isotropic case $a_1=\cdots=a_n=1$, this gives the corresponding formula for $P_tD_j$.
\end{lemma}

\begin{proof}
By linearity it suffices to test the identity on $g(\eps)=W_A(\eps)x$, where $x\in X$.  If $j\notin A$, both sides vanish, since $D_jW_A=0$ and $Z_j^a(t)$ is independent of $\prod_{i\in A}\xi_i^a(t)$ with mean zero.  If $j\in A$, then
\begin{equation}\label{equ-bsrP}
\E_\xi\bigl[Z_j^a(t)W_A(\eps\xi^a(t))x\bigr]
 =W_A(\eps)x\,\E[Z_j^a(t)\xi_j^a(t)]
   \prod_{i\in A\setminus\{j\}}\E\xi_i^a(t). 
\end{equation}
 
Since $\E[Z_j^a(t)\xi_j^a(t)]=\sqrt{1-e^{-2a_jt}}$, the right hand side of \eqref{equ-bsrP} becomes
\[
 W_A(\eps)x\,\sqrt{1-e^{-2a_jt}}
 e^{-t(a(A)-a_j)}.
\]
Multiplying by $e^{-a_jt}/\sqrt{1-e^{-2a_jt}}$ yields
$$e^{-ta(A)}W_A(\eps)x=P_t^aD_j(W_Ax)(\eps).$$
\end{proof}

The inverse operators may be represented by the semigroups on Walsh characters:
\[
 \Delta^{-1}D_j=
 \int_0^\infty P_tD_j\,dt,
 \qquad
 a_j\Delta_a^{-1}D_j=
 \int_0^\infty a_jP_t^aD_j\,dt.
\]
Indeed, both identities are immediate on each Walsh character and then extend by linearity.

\subsection{Square gradients and higher order square functions}

For $f:\Omega_n\to X$ define the first order square gradient by
\[
 |\nabla f|_X(\eps)=
 \left(\sum_{j=1}^n\|D_jf(\eps)\|_X^2\right)^{1/2}.
\]
For $k\geq1$, the ordered $k$-th square gradient is
\[
 G^{\mathrm{ord}}_{k,X}f(\eps)=
 \left(\sum_{i_1,\ldots,i_k=1}^n
 \|D_{i_1}\cdots D_{i_k}f(\eps)\|_X^2\right)^{1/2},
\]
and we set $G^{\mathrm{ord}}_{0,X}f(\eps)=\|f(\eps)\|_X$.  The distinct index square gradient is
\[
 G^{\mathrm{dis}}_{k,X}f(\eps)=
 \left(
 \sum_{\substack{A\subseteq\{1,\ldots,n\}\\ |A|=k}}
 \|D_Af(\eps)\|_X^2
 \right)^{1/2}.
\]
For $k=1$,
\[
 G^{\mathrm{ord}}_{1,X}f=G^{\mathrm{dis}}_{1,X}f=|\nabla f|_X.
\]
For every $k\geq1$, we use the pointwise comparison
\[
 G^{\mathrm{dis}}_{k,X}f\leq G^{\mathrm{ord}}_{k,X}f,
\]
since every distinct $k$-fold derivative appears among the ordered $k$-fold derivatives.

\subsection{Banach space conventions}

For $1\leq p<\infty$, the Bochner norm is
\[
 \|f\|_{L^p(\Omega_n;X)}:=
 \left(\E_\eps\|f(\eps)\|_X^p\right)^{1/p}.
\]
%When additional independent random signs occur, the expectation in the subscript indicates the underlying product probability space.  

We use $X^*$ for the dual Banach space and write $\langle x,x^*\rangle$ for the duality pairing.  The Walsh multipliers $D_j$, $\Delta^{-1}$, and $\Delta_a^{-1}$ are self-adjoint with respect to the scalar Walsh basis, hence, for appropriate $X$- and $X^*$-valued functions,
\[
 \E_\eps\langle Tf(\eps),h(\eps)\rangle
 =\E_\eps\langle f(\eps),Th(\eps)\rangle
\]
whenever $T$ is one of these real diagonal Walsh multipliers.

A Banach space $X$ has Rademacher type $2$ if there is a constant $T<\infty$ such that for every finite sequence $x_1,\ldots,x_m\in X$,
\[
 \left(\E_\eta\left\|\sum_{j=1}^m\eta_jx_j\right\|_X^2\right)^{1/2}
 \leq T\left(\sum_{j=1}^m\|x_j\|_X^2\right)^{1/2},
\]
where $\eta_1,\ldots,\eta_m$ are independent symmetric signs.  The least such $T$ is denoted by $T_2(X)$.

We shall use the Kahane--Khintchine inequality in the following standard form \cite{Kahane1985,LedouxTalagrand1991}: for every $0<p<\infty$ there is a constant $K_{2,p}<\infty$ such that, for every Banach space $X$ and every finite sequence $x_1,\ldots,x_m\in X$,
\[
 \left(\E_\eta\left\|\sum_{j=1}^m\eta_jx_j\right\|_X^2\right)^{1/2}
 \leq K_{2,p}
 \left(\E_\eta\left\|\sum_{j=1}^m\eta_jx_j\right\|_X^p\right)^{1/p}.
\]

Finally, for a finite index set $I$ we write $\ell_2^I(X)$ for the Hilbertian direct sum of $|I|$ copies of $X$, with norm
\[
 \|(x_i)_{i\in I}\|_{\ell_2^I(X)}=
 \left(\sum_{i\in I}\|x_i\|_X^2\right)^{1/2}.
\]
When $I=\{1,\ldots,M\}$ we also write $\ell_2^M(X)$.

\section{Randomized second order Riesz projections}
\label{sec:second-order-riesz}

Our first main result shows that a $\sqrt n$ bound of randomized second order Riesz projections for arbitrary Banach space. The proof is based on the observation that the second order Riesz projections admit a heat representation in which all Banach space dependence disappears.%; after this reduction, only scalar moments of a biased-sign sum remain. 
Throughout this section, $\delta=(\delta_1,\ldots,\delta_n)$ denotes an independent Rademacher vector, independent of all other variables.

\begin{theorem}[]
\label{thm:second-order-randomized}
Let $X$ be an arbitrary Banach space and let $1\leq p<\infty$.  Then for every $n\geq1$ and every $g:\Omega_n\to X$,
\begin{equation}
\label{eq:sqrt-n-dual}
 \left(\E_{\eps,\delta}
 \left\|
 \sum_{j=1}^n\delta_j\Delta^{-1}D_jg(\eps)
 \right\|_X^p\right)^{1/p}
 \leq C_p\sqrt n\,\|g\|_{L^p(\Omega_n;X)}.
\end{equation}
\end{theorem}

\begin{remark}[]
\label{rem:iv-relationship}
%In the notation of Ivanisvili--Volberg, inequality \eqref{eq:sqrt-n-dual} is their dual estimate (5.13) for the second-order Riesz projections $\Delta^{-1}D_j$; by duality, \eqref{eq:sqrt-n-primal} is their estimate (5.12).  

The Theorem 5.10 of \cite{IvanisviliVolberg2022} proves the bound $C(p,n)\lesssim_p n\log n$, and their Remark 5.11 suggests that the correct order should be $\sqrt n$ \cite{IvanisviliVolberg2022}.  Theorem~\ref{thm:second-order-randomized} verifies this fact for all Banach spaces.  The order $\sqrt n$ is also the right general dimension dependence: Ivanisvili--Volberg record, using the Hyt\"onen--Naor example, that in the finite cotype class the constant can grow at least as $\sqrt n$ \cite{HytonenNaor2013,IvanisviliVolberg2022}.  %Thus the theorem improves the general Banach-space upper bound while matching the known worst-order behavior.  

%It is complementary to the dimension-free estimates under nontrivial type assumptions in the same paper.
\end{remark}

\begin{proof}
For $t>0$, let $\xi(t)=(\xi_1(t),\ldots,\xi_n(t))$ be independent of $\eps$ and satisfy
\[
 \mathbb P\{\xi_i(t)=1\}=\frac{1+e^{-t}}2,
 \qquad
 \mathbb P\{\xi_i(t)=-1\}=\frac{1-e^{-t}}2.
\]
Put $Z_i(t)=\frac{\xi_i(t)-e^{-t}}{\sqrt{1-e^{-2t}}}.$ Then $\mathbb E Z_i(t)=0$ and $\mathbb E Z_i(t)^2=1$.  In the isotropic case $a_1=\cdots=a_n=1$, Lemma~\ref{lem:biased-sign-derivative} gives
\begin{equation}
\label{eq:isotropic-heat-derivative}
 P_tD_jg(\eps)
 =\frac{e^{-t}}{\sqrt{1-e^{-2t}}}
  \E_\xi\bigl[Z_j(t)g(\eps\xi(t))\bigr].
\end{equation}
Equivalently, \eqref{eq:isotropic-heat-derivative} can be checked directly on Walsh characters: if $g(\eps)=W_A(\eps)x$, then both sides vanish when $j\notin A$, while for $j\in A$,
\[
 \E_\xi[Z_j(t)W_A(\eps\xi(t))x]
 =W_A(\eps)x\,\E[Z_j(t)\xi_j(t)]
  \prod_{i\in A\setminus\{j\}}\E\xi_i(t)
 =W_A(\eps)x\sqrt{1-e^{-2t}}\,e^{-t(|A|-1)}.
\]
Multiplication by $e^{-t}/\sqrt{1-e^{-2t}}$ gives
$e^{-t|A|}W_A(\eps)x=P_tD_j(W_Ax)(\eps)$.

Since $\Delta^{-1}D_j=\int_0^\infty P_tD_j\,dt$ on every Walsh character, \eqref{eq:isotropic-heat-derivative} gives
\begin{equation}
\label{eq:second-order-heat-expansion}
 \sum_{j=1}^n\delta_j\Delta^{-1}D_jg(\eps)
 =\int_0^\infty b(t)\,
  \E_\xi\left[
  \left(\sum_{j=1}^n\delta_jZ_j(t)\right)g(\eps\xi(t))
  \right]dt,
\end{equation}
where $b(t)=\frac{e^{-t}}{\sqrt{1-e^{-2t}}}.$ For complete formal rigor, one may first integrate over $[a,R]$ with $0<a<R<\infty$, use the estimates below, and then let $a\downarrow0$ and $R\uparrow\infty$.

Define
\[
 M_p(t,n)=
 \left(\E_{\delta,\xi}\left|\sum_{j=1}^n\delta_jZ_j(t)\right|^p\right)^{1/p}.
\]
We claim that
\begin{equation}
\label{eq:Mp-integral-bound}
 \int_0^\infty b(t)M_p(t,n)\,dt\leq C_p\sqrt n.
\end{equation}
If $1\leq p\leq2$, then $M_p(t,n)\leq M_2(t,n)=\sqrt n$, because the random variables $\delta_jZ_j(t)$ are pairwise orthogonal in $L^2$.  Also
\[
 \int_0^\infty b(t)\,dt
 =\int_0^1\frac{du}{\sqrt{1-u^2}}
 =\frac\pi2,
\]
where $u=e^{-t}$.  Hence \eqref{eq:Mp-integral-bound} follows for $1\leq p\leq2$.

Assume now that $p>2$.  By the scalar Khintchine inequality, conditionally on $\xi(t)$,
\[
 \left(\E_\delta\left|\sum_{j=1}^n\delta_jZ_j(t)\right|^p\right)^{1/p}
 \leq K_p\left(\sum_{j=1}^n|Z_j(t)|^2\right)^{1/2}.
\]
Taking $L^p$ in $\xi$ and using Minkowski in $L^{p/2}$ gives
\[
 M_p(t,n)
 \leq K_p\left\|\sum_{j=1}^n|Z_j(t)|^2\right\|_{L^{p/2}(\xi)}^{1/2}
 \leq K_p\left(\sum_{j=1}^n\||Z_j(t)|^2\|_{L^{p/2}(\xi)}\right)^{1/2}
 =K_p\sqrt n\,\|Z_1(t)\|_{L^p}.
\]
It remains to check that
\[
 \int_0^\infty b(t)\|Z_1(t)\|_{L^p}\,dt<\infty.
\]
When $0<t\leq1$, the two possible values of $Z_1(t)$ satisfy
\[
 \left|\frac{1-e^{-t}}{\sqrt{1-e^{-2t}}}\right|\leq C\sqrt t,
 \qquad
 \left|\frac{-1-e^{-t}}{\sqrt{1-e^{-2t}}}\right|\leq Ct^{-1/2},
\]
and their probabilities are respectively $(1+e^{-t})/2$ and $(1-e^{-t})/2\leq Ct$.  Therefore
\[
 \|Z_1(t)\|_{L^p}^p\leq C_p(t^{p/2}+t^{1-p/2})\leq C_pt^{1-p/2}.
\]
Since $b(t)\leq Ct^{-1/2}$ on $(0,1]$, the integrand is bounded there by $C_pt^{1/p-1}$, which is integrable.  On $[1,\infty)$, $b(t)\leq Ce^{-t}$ and $\|Z_1(t)\|_{L^p}\leq C_p$.  This proves \eqref{eq:Mp-integral-bound} for every $1\leq p<\infty$.

We now estimate the operator.  From \eqref{eq:second-order-heat-expansion}, Minkowski's integral inequality and Jensen's inequality for the conditional expectation in $\xi$ yield
\[
\begin{aligned}
 &\left(\E_{\eps,\delta}
 \left\|
 \sum_{j=1}^n\delta_j\Delta^{-1}D_jg(\eps)
 \right\|_X^p\right)^{1/p}                                      \\
 &\quad\leq
 \int_0^\infty b(t)
 \left(
 \E_{\eps,\delta,\xi}
 \left|\sum_{j=1}^n\delta_jZ_j(t)\right|^p
 \|g(\eps\xi(t))\|_X^p
 \right)^{1/p}dt.
\end{aligned}
\]
For each fixed $\xi(t)$, the map $\eps\mapsto\eps\xi(t)$ preserves the uniform measure on $\Omega_n$.  Thus
\[
 \E_\eps\|g(\eps\xi(t))\|_X^p
 =\E_\eps\|g(\eps)\|_X^p,
\]
and the last display is at most
\[
 \|g\|_{L^p(\Omega_n;X)}
 \int_0^\infty b(t)M_p(t,n)\,dt
 \leq C_p\sqrt n\,\|g\|_{L^p(\Omega_n;X)}.
\]
This proves \eqref{eq:sqrt-n-dual}.
\end{proof}

\eqref{eq:sqrt-n-dual} is the dual form of the second order projection inequality in Ivanisvili--Volberg \cite{IvanisviliVolberg2022}. It's easy to show the following corollary.  %We record the corresponding primal formulation separately, since the statement is often the more convenient one when the input is a vector field.

\begin{corollary}[]
\label{cor:second-order-primal}
Let $1<p<\infty$, let $X$ be a Banach space, and let $F:\Omega_n\times\Omega_n\to X$.  Set $F_j(\eps)=\E_\delta[\delta_jF(\eps,\delta)],
 \qquad 1\leq j\leq n.$ Then
\begin{equation}
\label{eq:sqrt-n-primal}
 \left(\E_\eps
 \left\|
 \sum_{j=1}^n\Delta^{-1}D_jF_j(\eps)
 \right\|_X^p\right)^{1/p}
 \leq C_{p'}\sqrt n\,
 \left(\E_{\eps,\delta}\|F(\eps,\delta)\|_X^p\right)^{1/p},
\end{equation}
where $p'=p/(p-1)$.
\end{corollary}

\begin{proof}
Let $h\in L^{p'}(\Omega_n;X^*)$.  Since $D_j$ and $\Delta^{-1}$ are self-adjoint Walsh multipliers,
\[
\begin{aligned}
 \E_\eps\left\langle\sum_{j=1}^n\Delta^{-1}D_jF_j(\eps),h(\eps)\right\rangle
 &=\E_\eps\sum_{j=1}^n\left\langle F_j(\eps),\Delta^{-1}D_jh(\eps)\right\rangle       \\
 &=\E_{\eps,\delta}
 \left\langle F(\eps,\delta),
  \sum_{j=1}^n\delta_j\Delta^{-1}D_jh(\eps)
 \right\rangle .
\end{aligned}
\]
By Holder's inequality and Theorem~\ref{thm:second-order-randomized} applied to the Banach space $X^*$ and the exponent $p'$,
\[
\begin{aligned}
 \left|
 \E_\eps\left\langle\sum_{j=1}^n\Delta^{-1}D_jF_j,h\right\rangle
 \right|
 &\leq
 \|F\|_{L^p(\Omega_n\times\Omega_n;X)}
 \left\|
  \sum_{j=1}^n\delta_j\Delta^{-1}D_jh
 \right\|_{L^{p'}(\Omega_n\times\Omega_n;X^*)}       \\
 &\leq C_{p'}\sqrt n\,
 \|F\|_{L^p(\Omega_n\times\Omega_n;X)}
 \|h\|_{L^{p'}(\Omega_n;X^*)}.
\end{aligned}
\]
Taking the supremum over all such $h$ with $\|h\|_{L^{p'}(\Omega_n;X^*)}\leq1$ proves \eqref{eq:sqrt-n-primal}.
\end{proof}

\subsection*{Weighted anisotropic Riesz estimates}
\label{sec:weighted-anisotropic}

The proof of Theorem~\ref{thm:second-order-randomized} is insensitive to isotropy.  The same argument applies to weighted number operators, but the resulting constant is no longer governed only by the number of coordinates.  The inverse $\Delta_a^{-1}$ sees the full weighted frequency $a(A)=\sum_{j\in A}a_j$.

\begin{definition}
\label{def:lambda-pa}
Let $a=(a_1,\ldots,a_n)\in(0,\infty)^n$ and $1\leq p<\infty$.  With the notation of Lemma~\ref{lem:biased-sign-derivative}, set
\[
 \Lambda_p(a)=
 \int_0^\infty
 \left(
 \E_{\delta,\xi}
 \left|
 \sum_{j=1}^n
 \delta_j\frac{a_je^{-a_jt}}{\sqrt{1-e^{-2a_jt}}}
 Z_j^a(t)
 \right|^p
 \right)^{1/p}dt.
\]
\end{definition}

\begin{theorem}
\label{thm:weighted-riesz}
Let $a_1,\ldots,a_n>0$ and let $\Delta_a=\sum_{j=1}^n a_jD_j$.  Then $\Lambda_p(a)<\infty$ for every $1\leq p<\infty$, and for every Banach space $X$ and every $g:\Omega_n\to X$,
\begin{equation}
\label{eq:weighted-riesz}
 \left(
 \E_{\eps,\delta}
 \left\|
 \sum_{j=1}^n\delta_j a_j\Delta_a^{-1}D_jg(\eps)
 \right\|_X^p
 \right)^{1/p}
 \leq
 \Lambda_p(a)\,\|g\|_{L^p(\Omega_n;X)}.
\end{equation}
Moreover, if $1\leq p\leq2$, then
\begin{equation}
\label{eq:lambda2-bound}
 \Lambda_p(a)\leq\Lambda_2(a)
 =\int_0^\infty
 \left(\sum_{j=1}^n\frac{a_j^2}{e^{2a_jt}-1}\right)^{1/2}dt.
\end{equation}
In particular, if $a_1=\cdots=a_n$, then $\Lambda_2(a)=(\pi/2)\sqrt n$.
\end{theorem}

Note that for $a_1=\cdots=a_n$, Theorem~\ref{thm:weighted-riesz} reduces, in the range $1\leq p\leq2$, to the $\sqrt n$ estimate \eqref{eq:lambda2-bound} with the explicit constant $\pi/2$ at $p=2$; Theorem~\ref{thm:second-order-randomized} gives the same dimension order for all finite $p$.  

\begin{remark}[]
\label{rem:weighted-relationship}

For general weights, the constant is not determined only by the number of active coordinates.  The functional
\[
 \int_0^\infty
 \left(\sum_{j=1}^n\frac{a_j^2}{e^{2a_jt}-1}\right)^{1/2}dt
\]
captures the distribution of the scales $a_j^{-1}$.  Thus the Theorem \ref{thm:weighted-riesz} is a anisotropic extension of the second-order Riesz estimate, rather than the coordinate subset refinement appearing in the appendix of Ben-Efraim--Lust-Piquard \cite{BenEfraimLustPiquard2008}.  Their appendix corresponds to switching coordinates on or off for the unweighted number operator, whereas here the inverse $\Delta_a^{-1}$ couples all active coordinates through the weighted frequency $\sum_{j\in A}a_j$.
\end{remark}

\begin{proof}
Let $P_t^a=e^{-t\Delta_a}$.  By Lemma~\ref{lem:biased-sign-derivative}, for $t>0$ and $1\leq j\leq n$,
\begin{equation}
\label{eq:weighted-heat-derivative-again}
 P_t^aD_jg(\eps)
 =\frac{e^{-a_jt}}{\sqrt{1-e^{-2a_jt}}}
 \E_\xi\bigl[Z_j^a(t)g(\eps\xi^a(t))\bigr].
\end{equation}
Equivalently, testing on $g(\eps)=W_A(\eps)x$ gives the same identity directly.  If $j\notin A$, both sides vanish.  If $j\in A$, then
\[
 \E_\xi[Z_j^a(t)W_A(\eps\xi^a(t))x]
 =W_A(\eps)x\sqrt{1-e^{-2a_jt}}
  \prod_{i\in A\setminus\{j\}}e^{-a_it}.
\]
Multiplication by $e^{-a_jt}/\sqrt{1-e^{-2a_jt}}$ gives
$e^{-t\sum_{i\in A}a_i}W_A(\eps)x=P_t^aD_j(W_Ax)(\eps)$.

Since
\[
 a_j\Delta_a^{-1}D_j=
 \int_0^\infty a_jP_t^aD_j\,dt,
\]
we obtain
\begin{equation}
\label{eq:weighted-heat-expansion}
 \sum_{j=1}^n\delta_j a_j\Delta_a^{-1}D_jg(\eps)
 =\int_0^\infty
 \E_\xi\bigl[S_a(t,\delta,\xi)g(\eps\xi^a(t))\bigr]dt,
\end{equation}
where
\[
 S_a(t,\delta,\xi):=
 \sum_{j=1}^n
 \delta_j\frac{a_je^{-a_jt}}{\sqrt{1-e^{-2a_jt}}}
 Z_j^a(t).
\]
Minkowski's integral inequality, Jensen's inequality in $\xi$, and invariance of the uniform measure under $\eps\mapsto\eps\xi^a(t)$ give
\[
\begin{aligned}
 &\left(
 \E_{\eps,\delta}
 \left\|
 \sum_{j=1}^n\delta_j a_j\Delta_a^{-1}D_jg(\eps)
 \right\|_X^p
 \right)^{1/p}                                                    \\
 &\quad\leq
 \int_0^\infty
 \left(
 \E_{\eps,\delta,\xi}
 |S_a(t,\delta,\xi)|^p\|g(\eps\xi^a(t))\|_X^p
 \right)^{1/p}dt                                                   \\
 &\quad=
 \|g\|_{L^p(\Omega_n;X)}
 \int_0^\infty
 \left(\E_{\delta,\xi}|S_a(t,\delta,\xi)|^p\right)^{1/p}dt.
\end{aligned}
\]
The last integral is $\Lambda_p(a)$, proving \eqref{eq:weighted-riesz} once finiteness is known.

To prove finiteness, use the triangle inequality in $L^p$:
\[
 \left(\E|S_a(t,\delta,\xi)|^p\right)^{1/p}
 \leq
 \sum_{j=1}^n
 \frac{a_je^{-a_jt}}{\sqrt{1-e^{-2a_jt}}}
 \|Z_j^a(t)\|_{L^p}.
\]
After the change of variables $s=a_jt$, the $j$-th integral is
\[
 \int_0^\infty
 \frac{e^{-s}}{\sqrt{1-e^{-2s}}}\,
 \|Z_j(s)\|_{L^p}\,ds,
\]
where $Z_j(s)=(\xi_j(s)-e^{-s})/\sqrt{1-e^{-2s}}$ with
$\mathbb P\{\xi_j(s)=1\}=(1+e^{-s})/2$.  This integral is finite by the same estimate used in the proof of Theorem~\ref{thm:second-order-randomized}.  Hence $\Lambda_p(a)<\infty$.

If $1\leq p\leq2$, then $L^p\leq L^2$ gives $\Lambda_p(a)\leq\Lambda_2(a)$.  Finally, by independence and orthogonality in the auxiliary Rademacher signs $\delta_j$,
\[
 \E_{\delta,\xi}|S_a(t,\delta,\xi)|^2
 =\sum_{j=1}^n
 \frac{a_j^2e^{-2a_jt}}{1-e^{-2a_jt}}
 \E|Z_j^a(t)|^2
 =\sum_{j=1}^n\frac{a_j^2}{e^{2a_jt}-1},
\]
which proves \eqref{eq:lambda2-bound}.  If all $a_j$ are identical, say $a_j=a$, then
\[
 \Lambda_2(a)
 =\sqrt n\int_0^\infty\frac{ae^{-at}}{\sqrt{1-e^{-2at}}}\,dt
 =\sqrt n\int_0^\infty\frac{e^{-s}}{\sqrt{1-e^{-2s}}}\,ds
 =\frac\pi2\sqrt n.
\]
\end{proof}

\section{Type $2$ and fixed chaos square gradient lower Riesz estimates}
\label{sec:type2-fixed-chaos}

In this section we study the lower Riesz estimate with the pointwise square gradient
\begin{equation}
\label{eq:square-gradient-lower-riesz-global}
   \|\Delta^{1/2}f\|_{L^p(\Omega_n;X)}
   \le C\,\||\nabla f|_X\|_{L^p(\Omega_n)}.
\end{equation}
%The results below concern this square-gradient estimate, not the randomized-gradient lower Riesz problem of Ivanisvili--Volberg \cite{IvanisviliVolberg2022}.  

We can show that on each fixed homogeneous Walsh chaos, a Banach space condition for \eqref{eq:square-gradient-lower-riesz-global} is Rademacher type $2$.

For $1<p<\infty$ and $k\ge1$, define
\[
   \mathfrak L_{p,k}^{\rm hom}(X)
   :=
   \sup_{n\ge k}
   \sup_{0\ne f\in H_k(\Omega_n;X)}
   \frac{\|\Delta^{1/2}f\|_{L^p(\Omega_n;X)}}
        {\||\nabla f|_X\|_{L^p(\Omega_n)}}.
\]
Since $f\in H_k(\Omega_n;X)$ implies $\Delta^{1/2}f=\sqrt{k}\,f$, equivalently
\[
   \mathfrak L_{p,k}^{\rm hom}(X)
   =
   \sqrt{k}
   \sup_{n\ge k}
   \sup_{0\ne f\in H_k(\Omega_n;X)}
   \frac{\|f\|_{L^p(\Omega_n;X)}}
        {\||\nabla f|_X\|_{L^p(\Omega_n)}}.
\]

We shall use the following standard Banach-valued decoupling theorem for Rademacher chaoses.  It follows from the decoupling inequalities of de la Pe\~na and Montgomery-Smith for multivariate $U$-statistics \cite{deLaPenaMontgomerySmith1995}. % The two-sided form below is the only decoupling input needed in the proof.

\begin{lemma}
\label{lem:decoupling-chaos}
Let $r\ge1$ and $0<p<\infty$.  There is a constant $D_{p,r}<\infty$ such that, for every Banach space $E$ and every finite symmetric diagonal-free array $(a_{i_1,\ldots,i_r})\subset E$,
\begin{align*}
&\left\|
   \sum_{i_1,\ldots,i_r}^{\ne}
   a_{i_1,\ldots,i_r}\varepsilon_{i_1}\cdots\varepsilon_{i_r}
 \right\|_{L^p(E)}  \le
D_{p,r}
\left\|
   \sum_{i_1,\ldots,i_r}^{\ne}
   a_{i_1,\ldots,i_r}
   \varepsilon_{i_1}^{(1)}\cdots\varepsilon_{i_r}^{(r)}
\right\|_{L^p(E)},
\end{align*}
and conversely
\begin{align*}
&\left\|
   \sum_{i_1,\ldots,i_r}^{\ne}
   a_{i_1,\ldots,i_r}
   \varepsilon_{i_1}^{(1)}\cdots\varepsilon_{i_r}^{(r)}
\right\|_{L^p(E)} \le
D_{p,r}
\left\|
   \sum_{i_1,\ldots,i_r}^{\ne}
   a_{i_1,\ldots,i_r}\varepsilon_{i_1}\cdots\varepsilon_{i_r}
 \right\|_{L^p(E)}.
\end{align*}
Here $\sum^{\ne}$ means that the indices in each summand are pairwise distinct, and $\varepsilon^{(1)},\ldots,\varepsilon^{(r)}$ are independent Rademacher sequences.
\end{lemma}

\begin{theorem}
\label{thm:fixed-chaos-type2}
Let $1<p<\infty$ and $k\ge1$.  There exist constants $0<c_p\le C_{p,k}<\infty$ such that, for every Banach space $X$,
\begin{equation}
\label{eq:fixed-chaos-type2-equivalence}
   c_pT_2(X)
   \le
   \mathfrak L_{p,k}^{\rm hom}(X)
   \le
   C_{p,k}T_2(X),
\end{equation}
with the convention that the right-hand side is $+\infty$ if $X$ does not have Rademacher type $2$.
Consequently, for each fixed $k$, the estimate
\[
   \|\Delta^{1/2}f\|_{L^p(\Omega_n;X)}
   \le C\,\||\nabla f|_X\|_{L^p(\Omega_n)},
   \qquad f\in H_k(\Omega_n;X),\ n\ge k,
\]
holds with a constant independent of $n$ if and only if $X$ has Rademacher type $2$.
\end{theorem}

\begin{proof}
We first prove the lower bound in \eqref{eq:fixed-chaos-type2-equivalence}.  This is the block construction giving the type $2$ obstruction.  Fix $x_1,\ldots,x_m\in X$ and partition the coordinates of $\Omega_{mk}$ into disjoint blocks
\[
   B_r:=\{(r-1)k+1,\ldots,rk\},
   \qquad 1\le r\le m.
\]
Let $\rho_r(\varepsilon)=\prod_{i\in B_r}\varepsilon_i$, and $f(\varepsilon)=\sum_{r=1}^m\rho_r(\varepsilon)x_r.$ The random variables $\rho_1,\ldots,\rho_m$ are independent Rademacher variables, and $f\in H_k(\Omega_{mk};X)$.  Hence
\[
   \Delta^{1/2}f=\sqrt{k}\,f.
\]
If $i\in B_r$, then $D_if=\rho_r x_r$, while the other block contributions vanish.  Therefore
\[
   |\nabla f|_X(\varepsilon)
   =
   \left(\sum_{r=1}^m\sum_{i\in B_r}
   \|\rho_r(\varepsilon)x_r\|_X^2\right)^{1/2}
   =
   \sqrt{k}\left(\sum_{r=1}^m\|x_r\|_X^2\right)^{1/2}.
\]
Consequently,
\[
   \mathfrak L_{p,k}^{\rm hom}(X)
   \ge
   \frac{
   \left(\mathbb E\left\|\sum_{r=1}^m\rho_r x_r\right\|_X^p\right)^{1/p}}
   {\left(\sum_{r=1}^m\|x_r\|_X^2\right)^{1/2}}.
\]
Taking the supremum over all finite sequences $(x_r)$ and using the Kahane--Khintchine inequality gives
\[
   \mathfrak L_{p,k}^{\rm hom}(X)
   \ge c_pT_2(X).
\]

We now prove the upper bound.  The case $k=1$ follows directly from type $2$ and Kahane--Khintchine.  Indeed, if
\[
   f(\varepsilon)=\sum_{j=1}^n\varepsilon_jx_j\in H_1(\Omega_n;X),
\]
then $\Delta^{1/2}f=f$ and
\[
   |\nabla f|_X(\varepsilon)=\left(\sum_{j=1}^n\|x_j\|_X^2\right)^{1/2}.
\]
Thus
\[
   \|\Delta^{1/2}f\|_{L^p(\Omega_n;X)}
   =\|f\|_{L^p(\Omega_n;X)}
   \le K_pT_2(X)
   \left(\sum_{j=1}^n\|x_j\|_X^2\right)^{1/2}
   =K_pT_2(X)\||\nabla f|_X\|_{L^p(\Omega_n)}.
\]

Assume from now on that $k\ge2$, and write
\[
   f(\varepsilon)=\sum_{|A|=k}W_A(\varepsilon)x_A\in H_k(\Omega_n;X).
\]
Define a symmetric diagonal free ordered array by
\[
   a_{i_1,\ldots,i_k}
   :=
   \begin{cases}
      x_{\{i_1,\ldots,i_k\}}, & i_1,\ldots,i_k \text{ pairwise distinct},\\
      0, & \text{otherwise}.
   \end{cases}
\]
Then
\[
   Q(\varepsilon)
   =
   \sum_{i_1,\ldots,i_k}^{\ne}
   a_{i_1,\ldots,i_k}\varepsilon_{i_1}\cdots\varepsilon_{i_k}
   =k!\,f(\varepsilon).
\]
Let
\[
   \widetilde Q(\varepsilon^{(1)},\ldots,\varepsilon^{(k)})
   :=
   \sum_{i_1,\ldots,i_k}^{\ne}
   a_{i_1,\ldots,i_k}
   \varepsilon_{i_1}^{(1)}\cdots\varepsilon_{i_k}^{(k)}.
\]
By Lemma~\ref{lem:decoupling-chaos},
\begin{equation}
\label{eq:first-decoupling-Q}
   k!\,\|f\|_{L^p(\Omega_n;X)}
   =\|Q\|_{L^p(X)}
   \le D_{p,k}\|\widetilde Q\|_{L^p(X)}.
\end{equation}
Separate the first family of signs:
\[
   \widetilde Q
   =
   \sum_{j=1}^n\varepsilon_j^{(1)}
   H_j(\varepsilon^{(2)},\ldots,\varepsilon^{(k)}),
\]
where
\[
   H_j
   =
   \sum_{\substack{i_2,\ldots,i_k\\ j,i_2,\ldots,i_k\ \text{pairwise distinct}}}
   \varepsilon_{i_2}^{(2)}\cdots\varepsilon_{i_k}^{(k)}
   x_{\{j,i_2,\ldots,i_k\}}.
\]
Conditioning on $\varepsilon^{(2)},\ldots,\varepsilon^{(k)}$, type $2$ and Kahane--Khintchine give
\[
   \left(\mathbb E_{\varepsilon^{(1)}}
   \left\|\sum_{j=1}^n\varepsilon_j^{(1)}H_j\right\|_X^p
   \right)^{1/p}
   \le
   K_pT_2(X)
   \left(\sum_{j=1}^n\|H_j\|_X^2\right)^{1/2}.
\]
Integrating in the remaining signs yields
\begin{equation}
\label{eq:type2-H-vector}
   \|\widetilde Q\|_{L^p(X)}
   \le
   K_pT_2(X)\,
   \|(H_1,\ldots,H_n)\|_{L^p(\ell_2^n(X))}.
\end{equation}

It remains to compare the decoupled vector $(H_1,\ldots,H_n)$ with the gradient of $f$.  Apply Lemma~\ref{lem:decoupling-chaos} in the Banach space $\ell_2^n(X)$ at order $k-1$ to the symmetric diagonal-free array
\[
   b_{i_2,\ldots,i_k}:=\big((b_{i_2,\ldots,i_k})_j\big)_{j=1}^n
   \in \ell_2^n(X),
\]
where
\[
   (b_{i_2,\ldots,i_k})_j
   :=
   \begin{cases}
      x_{\{j,i_2,\ldots,i_k\}}, & j,i_2,\ldots,i_k \text{ are pairwise distinct},\\
      0, & \text{otherwise}.
   \end{cases}
\]
Its decoupled chaos is $(H_1,\ldots,H_n)$, and its coupled chaos is $(H_1^c,\ldots,H_n^c)$, where
\[
   H_j^c(\eta)
   :=
   \sum_{\substack{i_2,\ldots,i_k\\ j,i_2,\ldots,i_k\ \text{pairwise distinct}}}
   \eta_{i_2}\cdots\eta_{i_k}
   x_{\{j,i_2,\ldots,i_k\}}.
\]
For every $j$ and every $A\subset\{1,\ldots,n\}$ with $j\in A$ and $|A|=k$, the ordered tuples $(i_2,\ldots,i_k)$ enumerating $A\setminus\{j\}$ occur $(k-1)!$ times.  Since our Walsh derivative is the multiplier $D_jW_A=\mathbf 1_{\{j\in A\}}W_A$, we have
\[
   H_j^c(\eta)=(k-1)!\,\eta_jD_jf(\eta).
\]
As $|\eta_j|=1$,
\[
   \|(H_1^c,\ldots,H_n^c)\|_{L^p(\ell_2^n(X))}
   =(k-1)!\,\||\nabla f|_X\|_{L^p(\Omega_n)}.
\]
The second decoupling step therefore gives
\begin{equation}
\label{eq:second-decoupling-H}
   \|(H_1,\ldots,H_n)\|_{L^p(\ell_2^n(X))}
   \le
   D_{p,k-1}(k-1)!\,
   \||\nabla f|_X\|_{L^p(\Omega_n)}.
\end{equation}
Combining \eqref{eq:first-decoupling-Q}, \eqref{eq:type2-H-vector}, and \eqref{eq:second-decoupling-H}, we obtain
\[
   k!\,\|f\|_{L^p(\Omega_n;X)}
   \le
   D_{p,k}K_pT_2(X)D_{p,k-1}(k-1)!\,
   \||\nabla f|_X\|_{L^p(\Omega_n)}.
\]
Since $\Delta^{1/2}f=\sqrt{k}\,f$, this implies
\[
   \|\Delta^{1/2}f\|_{L^p(\Omega_n;X)}
   \le
   \frac{D_{p,k}K_pD_{p,k-1}}{\sqrt{k}}\,
   T_2(X)\,
   \||\nabla f|_X\|_{L^p(\Omega_n)}.
\]
The upper bound follows after increasing the constant to cover $k=1$.
\end{proof}

\begin{corollary}
\label{cor:global-square-gradient-forces-type2}
Let $1<p<\infty$.  Suppose that there is a constant $C_p<\infty$ such that \eqref{eq:square-gradient-lower-riesz-global} holds for every $n\ge1$ and every $f:\Omega_n\to X$.  Then $X$ has Rademacher type $2$ and
\[
   T_2(X)\le c_p^{-1}C_p,
\]
where $c_p$ is the constant from Theorem~\ref{thm:fixed-chaos-type2}.
\end{corollary}

\begin{proof}
The global estimate implies the same estimate on $H_1(\Omega_n;X)$ for every $n$.  Applying the lower bound half of Theorem~\ref{thm:fixed-chaos-type2} with $k=1$ gives the claim.
\end{proof}

\begin{corollary}
\label{cor:finite-cotype-not-sufficient}
Finite cotype is not sufficient for the dimension free pointwise square gradient lower Riesz estimate \eqref{eq:square-gradient-lower-riesz-global}.  In fact, cotype $2$ is not sufficient.
\end{corollary}

\begin{proof}
For $1<r<2$, the space $\ell_r$ has cotype $2$ but does not have Rademacher type $2$.  Hence Corollary~\ref{cor:global-square-gradient-forces-type2} already implies that \eqref{eq:square-gradient-lower-riesz-global} cannot hold on $\ell_r$ with a dimension free constant.

The failure is visible in the first Walsh chaos.  Let $e_1,\ldots,e_n$ be the standard unit vectors of $\ell_r$ and set
\[
   f(\varepsilon):=\sum_{j=1}^n\varepsilon_je_j.
\]
Then $f\in H_1(\Omega_n;\ell_r)$, $\Delta^{1/2}f=f$, and
\[
   \|\Delta^{1/2}f\|_{L^p(\Omega_n;\ell_r)}
   =
   \left(\mathbb E_\varepsilon
   \left\|\sum_{j=1}^n\varepsilon_je_j\right\|_{\ell_r}^p
   \right)^{1/p}
   =n^{1/r}.
\]
On the other hand,
\[
   \||\nabla f|_{\ell_r}\|_{L^p(\Omega_n)}
   =
   \left(\sum_{j=1}^n\|e_j\|_{\ell_r}^2\right)^{1/2}
   =n^{1/2}.
\]
Thus a dimension free estimate would imply $n^{1/r}\le Cn^{1/2}$ for all $n$, which is impossible because $1/r>1/2$.
\end{proof}

\begin{remark}[]
\label{rem:randomized-vs-square-gradient}
The corollary \ref{cor:finite-cotype-not-sufficient} concerns only the pointwise square gradient estimate \eqref{eq:square-gradient-lower-riesz-global}.  It is not a counterexample to the randomized gradient lower Riesz problem of Ivanisvili--Volberg \cite{IvanisviliVolberg2022}.  For example, in the test function used above,
\[
   f(\varepsilon)=\sum_{j=1}^n\varepsilon_je_j\in \ell_r,
   \qquad 1<r<2,
\]
the square gradient norm equals $n^{1/2}$, whereas the randomized gradient norm satisfies
\[
   \left(\mathbb E_\delta
   \left\|\sum_{j=1}^n\delta_jD_jf\right\|_{L^p(\Omega_n;\ell_r)}^p
   \right)^{1/p}
   =
   \left(\mathbb E_{\delta,\varepsilon}
   \left\|\sum_{j=1}^n\delta_j\varepsilon_je_j\right\|_{\ell_r}^p
   \right)^{1/p}
   =n^{1/r},
\]
which is of the same order as $\|\Delta^{1/2}f\|_{L^p(\Omega_n;\ell_r)}$. % Thus this example separates the pointwise square-gradient estimate from the randomized-gradient estimate, rather than disproving the latter.
\end{remark}

\section{Sharpness of the analytic paraproduct tail bound}
\label{sec:paraproduct-sharpness}

We now consider the one variable analytic obstruction used in Volberg's tail space argument \cite{Volberg2024}.  We can show that under a standard corner asymptotics, the paraproduct tail operator has norm of order $d^{-\alpha}$ on tail spaces. % Thus that analytic step cannot be improved to order $d^{-1}$ when $\alpha<1$.

Let $Y$ be a complex Banach space.  For an integer $d\geq1$ define
\[
 H_d^\infty(\D;Y)
 :=\{g\in H^\infty(\D;Y): g^{(j)}(0)=0\text{ for }0\leq j<d\}.
\]
Equivalently, $g\in H_d^\infty(\D;Y)$ if and only if its Taylor expansion starts at degree at least $d$.

Let $\varphi\in H^\infty(\D)$ be nonconstant and put
\[
 T_\varphi g(z):=\int_0^z g(\zeta)\varphi'(\zeta)\,d\zeta,
 \qquad z\in\D.
\]
The integral is path independent because the integrand is holomorphic.

\begin{lemma}
\label{lem:tail-integration}
Let $K\geq0$ be an integer and let $h\in H^\infty(\D;Y)$ satisfy
$h^{(j)}(0)=0$ for $0\leq j<K$.  Define $Ih(z):=\int_0^z h(\zeta)\,d\zeta.$ Then
\[
 \|Ih\|_{H^\infty(\D;Y)}
 \leq \frac{1}{K+1}\|h\|_{H^\infty(\D;Y)}.
\]
\end{lemma}

\begin{proof}
Write $h(z)=z^Kq(z)$, where $q\in H(\D;Y)$.  We first check that
\[
 \|q\|_{H^\infty(\D;Y)}\leq\|h\|_{H^\infty(\D;Y)}.
\]
Indeed, for $0<r<1$, the maximum principle applied to $q$ on $|z|<r$ gives
\[
 \sup_{|z|\leq r}\|q(z)\|_Y
 \leq r^{-K}\sup_{|z|\leq r}\|h(z)\|_Y
 \leq r^{-K}\|h\|_{H^\infty(\D;Y)}.
\]
Letting $r\uparrow1$ after fixing $z$ proves the claim.

For $z\in\D$, integrate over the segment $[0,z]$:
\[
\begin{aligned}
 \|Ih(z)\|_Y
 &=\left\|\int_0^1h(tz)z\,dt\right\|_Y                                      \\
 &\leq\int_0^1|z|\,\|h(tz)\|_Y\,dt                                      \\
 &=\int_0^1|z|\,|tz|^K\,\|q(tz)\|_Y\,dt                                  \\
 &\leq\|q\|_{H^\infty(\D;Y)}|z|^{K+1}\int_0^1t^K\,dt                    \\
 &\leq\frac{1}{K+1}\|h\|_{H^\infty(\D;Y)}.
\end{aligned}
\]
Taking the supremum over $z\in\D$ proves the lemma.
\end{proof}

\begin{lemma}
\label{lem:coef-sum}
If $0<\alpha<1$, then there is a constant $C_\alpha<\infty$ such that for every integer $d\geq1$,
\[
 \sum_{m=1}^\infty\frac{m^{-\alpha}}{d+m}\leq C_\alpha d^{-\alpha}.
\]
\end{lemma}

\begin{proof}
Split the sum at $m=d$:
\[
 \sum_{m=1}^d\frac{m^{-\alpha}}{d+m}
 \leq\frac1d\sum_{m=1}^d m^{-\alpha}
 \leq\frac1d\left(1+\int_1^d x^{-\alpha}\,dx\right)
 \leq C_\alpha d^{-\alpha},
\]
and
\[
 \sum_{m=d+1}^\infty\frac{m^{-\alpha}}{d+m}
 \leq\sum_{m=d+1}^\infty m^{-\alpha-1}
 \leq\int_d^\infty x^{-\alpha-1}\,dx
 =\frac1\alpha d^{-\alpha}.
\]
Combining the two estimates proves the claim.
\end{proof}

\begin{theorem}
\label{thm:paraproduct-sharp}
Let $0<\alpha<1$ and let
\[
 \varphi(z)=\sum_{m=1}^\infty c_mz^m
\]
be a bounded nonconstant holomorphic function on $\D$.  Assume that the following two conditions hold.
\begin{enumerate}
\item[(A1)] There is $A<\infty$ such that
\[
 |c_m|\leq Am^{-1-\alpha}\qquad(m\geq1).
\]
\item[(A2)] There is $b\in\mathbb C\setminus\{0\}$ such that, as $r\uparrow1$,
\[
 \varphi'(r)=b(1-r)^{\alpha-1}+o\bigl((1-r)^{\alpha-1}\bigr).
\]
\end{enumerate}
Then, for every nonzero complex Banach space $Y$, there are constants
$0<c_{\alpha,\varphi}\leq C_{\alpha,\varphi}<\infty$ such that for every integer $d\geq1$,
\[
 c_{\alpha,\varphi}d^{-\alpha}
 \leq
 \|T_\varphi:H_d^\infty(\D;Y)\to H^\infty(\D;Y)\|
 \leq
 C_{\alpha,\varphi}d^{-\alpha}.
\]
Equivalently,
\[
 \|T_\varphi:H_d^\infty(\D;Y)\to H^\infty(\D;Y)\|
 \asymp_{\alpha,\varphi} d^{-\alpha}.
\]
\end{theorem}

%\begin{remark}[]
%\label{rem:analytic-corners}
If $\varphi$ is the conformal map from $\D$ onto a Jordan domain whose boundary consists, near $\varphi(1)$, of two analytic arcs meeting with interior angle $\pi\alpha$, $0<\alpha<1$, then the classical Lehman--Warschawski corner expansion gives
\[
 \varphi'(r)=b(1-r)^{\alpha-1}+o((1-r)^{\alpha-1})
\]
for some $b\neq0$ \cite{Lehman1957,Warschawski1942}.  In Volberg's argument, the coefficient estimate $|c_m|\lesssim m^{-1-\alpha}$ is precisely the estimate for the conformal map onto his domain $O_\alpha$ \cite{Volberg2024}.  Thus Theorem~\ref{thm:paraproduct-sharp} says that the analytic paraproduct step cannot, by itself, improve the factor $d^{-\alpha}$ to $d^{-1}$ when $\alpha<1$.
%\end{remark}

\begin{proof}
We prove the upper bound first.  Fix $g\in H_d^\infty(\D;Y)$.  Since
\[
 \varphi'(z)=\sum_{m=1}^\infty mc_mz^{m-1}\qquad(z\in\D),
\]
we have, initially uniformly on compact subsets of $\D$,
\[
 T_\varphi g(z)
 =\sum_{m=1}^\infty mc_m\int_0^z\zeta^{m-1}g(\zeta)\,d\zeta.
\]
For each $m\geq1$, the function $\zeta\mapsto\zeta^{m-1}g(\zeta)$ belongs to
$H_{d+m-1}^\infty(\D;Y)$ and has $H^\infty$ norm at most $\|g\|_{H^\infty(\D;Y)}$.  By Lemma~\ref{lem:tail-integration},
\[
 \left\|\int_0^z\zeta^{m-1}g(\zeta)\,d\zeta\right\|_{H^\infty(\D;Y)}
 \leq\frac{1}{d+m}\|g\|_{H^\infty(\D;Y)}.
\]
Therefore, using (A1) and Lemma~\ref{lem:coef-sum},
\[
\begin{aligned}
 \|T_\varphi g\|_{H^\infty(\D;Y)}
 &\leq\sum_{m=1}^\infty m|c_m|\,
  \left\|\int_0^z\zeta^{m-1}g(\zeta)\,d\zeta\right\|_{H^\infty(\D;Y)}       \\
 &\leq A\|g\|_{H^\infty(\D;Y)}
  \sum_{m=1}^\infty\frac{m^{-\alpha}}{d+m}                              \\
 &\leq AC_\alpha d^{-\alpha}\|g\|_{H^\infty(\D;Y)}.
\end{aligned}
\]
This proves the upper bound.  The same estimate also justifies convergence of the displayed series in the operator norm relevant to the argument.

We now prove the lower bound.  Choose $y\in Y$ with $\|y\|_Y=1$ and set $g_d(z)=z^dy.$ Then $g_d\in H_d^\infty(\D;Y)$ and $\|g_d\|_{H^\infty(\D;Y)}=1$.  For $0<\rho<1$,
\[
 T_\varphi g_d(\rho)=y\int_0^\rho r^d\varphi'(r)\,dr.
\]
Condition (A2) implies that $\varphi'$ is integrable on $(0,1)$: near $1$ it is bounded by a constant multiple of $(1-r)^{\alpha-1}$, and $\alpha>0$.  Hence, by letting $\rho\uparrow1$,
\[
 \|T_\varphi g_d\|_{H^\infty(\D;Y)}
 \geq\left|\int_0^1 r^d\varphi'(r)\,dr\right|.
\]
We claim that
\[
 \int_0^1 r^d\varphi'(r)\,dr
 =bB(d+1,\alpha)+o(d^{-\alpha}),
 \qquad d\to\infty,
\]
where $B$ is the beta function.  To see this, fix $\eta>0$.  By (A2), choose $\delta\in(0,1)$ so that
\[
 |\varphi'(r)-b(1-r)^{\alpha-1}|
 \leq\eta(1-r)^{\alpha-1},
 \qquad 1-\delta<r<1.
\]
On $[0,1-\delta]$, the contribution of $\int_0^{1-\delta}r^d\varphi'(r)\,dr$ is $O((1-\delta)^d)$.  On $[1-\delta,1]$, the preceding inequality gives an error at most
\[
 \eta\int_{1-\delta}^1r^d(1-r)^{\alpha-1}\,dr
 \leq\eta B(d+1,\alpha).
\]
The missing part of the beta integral over $[0,1-\delta]$ is also $O((1-\delta)^d)$.  Since
\[
 B(d+1,\alpha)=\int_0^1 r^d(1-r)^{\alpha-1}\,dr
 =\frac{\Gamma(d+1)\Gamma(\alpha)}{\Gamma(d+1+\alpha)}
 \sim\Gamma(\alpha)d^{-\alpha},
\]
and since $\eta>0$ is arbitrary, the claim follows.  Thus
\[
 \int_0^1r^d\varphi'(r)\,dr
 =b\Gamma(\alpha)d^{-\alpha}+o(d^{-\alpha}).
\]
Consequently, for all sufficiently large $d$,
\[
 \|T_\varphi:H_d^\infty(\D;Y)\to H^\infty(\D;Y)\|
 \geq\|T_\varphi g_d\|_{H^\infty(\D;Y)}
 \geq\frac{|b|\Gamma(\alpha)}2d^{-\alpha}.
\]
For the finitely many remaining $d$, the operator norm is positive: indeed, $T_\varphi(z^dy)$ is not identically zero because its derivative is $z^d\varphi'(z)y$ and $\varphi'$ is not identically zero.  Decreasing the constant if necessary gives the lower bound for every $d\geq1$.
\end{proof}

\subsection*{Iteration of Bernstein--Markov estimates}
\label{sec:bm-iteration}

We show that once the first order Bernstein--Markov estimate is known uniformly for finite Hilbertian sums of the target space, the higher order Hilbertian square functions follow by iteration.  %The proof uses only the multiplier property of the derivatives and contains no additional Banach-space input.

\begin{theorem}
\label{thm:bm-iteration}
Fix $1<p<\infty$, $a\geq0$, and $C<\infty$.  Let $X$ be a Banach space.  Assume that, for every $M,n,d\geq1$ and every
$F\in\mathcal P_d(\Omega_n;\ell_2^M(X))$,
\begin{equation}
\label{eq:BM1}
 \|G^{\mathrm{ord}}_{1,\ell_2^M(X)}F\|_{L^p(\Omega_n)}
 \leq Cd^a\|F\|_{L^p(\Omega_n;\ell_2^M(X))}.
\end{equation}
Then, for every integer $k\geq1$, every $n,d\geq1$, and every $f\in\mathcal P_d(\Omega_n;X)$,
\[
 \|G^{\mathrm{ord}}_{k,X}f\|_{L^p(\Omega_n)}
 \leq C^k d^{ka}\|f\|_{L^p(\Omega_n;X)}.
\]
Consequently,
\[
 \|G^{\mathrm{dis}}_{k,X}f\|_{L^p(\Omega_n)}
 \leq C^k d^{ka}\|f\|_{L^p(\Omega_n;X)}.
\]
\end{theorem}

\begin{proof}
We prove the ordered estimate by induction on $k$.  The case $k=1$ is \eqref{eq:BM1} with $M=1$.  Assume that the estimate is known at order $k-1$.  Let $I_{k-1}=\{1,\ldots,n\}^{k-1},$ $Y=\ell_2^{I_{k-1}}(X),$ and define the $Y$-valued polynomial
\[
 F(\eps)=\bigl(D_{i_1}\cdots D_{i_{k-1}}f(\eps)\bigr)_{(i_1,\ldots,i_{k-1})\in I_{k-1}}.
\]
Each $D_i$ is a Walsh multiplier by $0$ or $1$, hence it does not increase degree; therefore $F\in\mathcal P_d(\Omega_n;Y)$.  Moreover, pointwise in $\eps$,
\[
 \|F(\eps)\|_Y=G^{\mathrm{ord}}_{k-1,X}f(\eps),
\]
and
\[
\begin{aligned}
 G^{\mathrm{ord}}_{1,Y}F(\eps)
 &=\left(
 \sum_{j=1}^n
 \sum_{(i_1,\ldots,i_{k-1})\in I_{k-1}}
 \|D_jD_{i_1}\cdots D_{i_{k-1}}f(\eps)\|_X^2
 \right)^{1/2}                                            \\
 &=G^{\mathrm{ord}}_{k,X}f(\eps).
\end{aligned}
\]
Applying \eqref{eq:BM1} to $F$ and then the induction hypothesis gives
\[
\begin{aligned}
 \|G^{\mathrm{ord}}_{k,X}f\|_{L^p}
 &=\|G^{\mathrm{ord}}_{1,Y}F\|_{L^p}                                      \\
 &\leq Cd^a\|F\|_{L^p(\Omega_n;Y)}                                      \\
 &=Cd^a\|G^{\mathrm{ord}}_{k-1,X}f\|_{L^p}                              \\
 &\leq C^kd^{ka}\|f\|_{L^p(\Omega_n;X)}.
\end{aligned}
\]
The estimate for $G^{\mathrm{dis}}_{k,X}$ follows from the pointwise inequality
$G^{\mathrm{dis}}_{k,X}f\leq G^{\mathrm{ord}}_{k,X}f$ in Section~\ref{sec:preliminaries}.
\end{proof}

\begin{corollary}
\label{cor:volberg-higher-order-bm}
Let $1<p<\infty$ and assume that $X^*$ has Rademacher type $2$.  Let $\alpha\in(0,1]$ be the Pisier holomorphic semigroup exponent, chosen uniformly for the finite Hilbertian sums $\ell_2^M(X)$.  Define
\[
 a_{p,X}=
 \begin{cases}
  (2-\alpha)/p, & 1<p<2,\\
  1-\alpha/2, & 2\leq p<\infty.
 \end{cases}
\]
Then, for every $k\geq1$, every $n,d\geq1$, and every $f\in\mathcal P_d(\Omega_n;X)$,
\[
 \|G^{\mathrm{ord}}_{k,X}f\|_{L^p(\Omega_n)}
 \leq C_{p,X}^k d^{ka_{p,X}}\|f\|_{L^p(\Omega_n;X)}.
\]
In particular,
\[
 \left\|
 \left(
 \sum_{\substack{A\subseteq\{1,\ldots,n\}\\ |A|=k}}
 \|D_Af\|_X^2
 \right)^{1/2}
 \right\|_{L^p(\Omega_n)}
 \leq C_{p,X}^k d^{ka_{p,X}}\|f\|_{L^p(\Omega_n;X)}.
\]
\end{corollary}

%\begin{remark}
%The corollary \eqref{cor:volberg-higher-order-bm} is an iteration statement, its content is that no additional Banach-space geometry is needed to pass from the first-order estimate for $|\nabla f|_X$ to higher-order Hilbertian square functions: the only extra assumption is uniformity under finite $\ell_2$-sums, which is built into the proof under the hypothesis that $X^*$ has type $2$.
%\end{remark}

\begin{proof}
Volberg's Theorem 4.1 gives the first order estimate with exponent $a_{p,X}$ for $X$-valued functions \cite{Volberg2024}.  The same proof applies to $Y_M:=\ell_2^M(X)$ with constants independent of $M$.  Indeed,
$Y_M^*=\ell_2^M(X^*)$, and if $X^*$ has constant $T$, then $Y_M^*$ has the same type-$2$ constant: for $u_j=(u_{j,m})_{m=1}^M\in\ell_2^M(X^*)$,
\[
\begin{aligned}
 \E_\eps\left\|\sum_j\eps_ju_j\right\|_{\ell_2^M(X^*)}^2
 &=\sum_{m=1}^M
 \E_\eps\left\|\sum_j\eps_ju_{j,m}\right\|_{X^*}^2                    \\
 &\leq T^2\sum_j\sum_{m=1}^M\|u_{j,m}\|_{X^*}^2.
\end{aligned}
\]
Also, the $K$-convexity constants entering Pisier's holomorphic semigroup theorem are stable under finite Hilbertian sums: with the usual $L^2$ definition of $K$-convexity, the Rademacher projection acts coordinatewise on $L^2(\Omega;\ell_2^M(X))$, which is isometric to $\ell_2^M(L^2(\Omega;X))$.  Hence the Pisier angle and the constants in Volberg's first order estimate may be chosen uniformly in $M$.

Thus assumption \eqref{eq:BM1} of Theorem~\ref{thm:bm-iteration} holds with $C=C_{p,X}$ and $a=a_{p,X}$.  The claimed estimates follow immediately from that theorem.
\end{proof}

\section*{Acknowledgments}
The authors would like to express their gratitude to Alexander Volberg for his encouragement and generous support. Yiming Chen was supported by  the New Cornerstone Science Foundation (4801500068).

\end{document}